\documentclass[12pt]{amsart}

\usepackage{graphicx,amssymb,latexsym,amsfonts,txfonts,amsmath,amsthm}
\usepackage{pdfsync,color,tabularx,rotating}
\usepackage[all,cmtip]{xy}
\usepackage{url}
\usepackage{tikz}
\usepackage{amssymb}
\usepackage{comment}

\advance\hoffset by -0.9 truecm

\newtheorem{theo}{Theorem}[section]
\newtheorem{prop}[theo]{Proposition}
\newtheorem{coro}[theo]{Corollary}
\newtheorem{lemm}[theo]{Lemma}
\newtheorem{conj}[theo]{Conjecture}

\theoremstyle{remark}
\newtheorem{rema}[theo]{\bf Remark}

\theoremstyle{remark}
\newtheorem{exam}[theo]{\bf Example}

\theoremstyle{remark}

\begin{document}

\title{Block designs and prime values of polynomials}

\author{Gareth A. Jones and Alexander K. Zvonkin}

\address{School of Mathematical Sciences, University of Southampton, Southampton SO17 1BJ, UK}
\email{G.A.Jones@maths.soton.ac.uk}

\address{LaBRI, Universit\'e de Bordeaux, 351 Cours de la Lib\'eration, F-33405 Talence Cedex, France}

\email{zvonkin@labri.fr}




\begin{abstract}
A recent construction by Amarra, Devillers and Praeger of block designs with specific parameters depends on certain quadratic polynomials, with integer coefficients, taking prime power values. The Bunyakovsky Conjecture, if true, would imply that each of them takes infinitely many prime values, giving an infinite family of block designs with the required parameters. We have found large numbers of prime values of these polynomials, and the numbers found agree very closely with the estimates for them provided by Li's recent modification of the Bateman--Horn Conjecture. While this does not prove that these polynomials take infinitely many prime values, it provides strong evidence for this, and it also adds extra support for the validity of the Bunyakovsky and Bateman-Horn Conjectures.
\end{abstract}

\maketitle


\section{Introduction}

Block designs play a major role in combinatorics and finite geometry, and have many applications in statistics, specifically in the design of experiments. In~\cite{ADP}, Amarra, Devillers and Praeger have recently constructed families of highly symmetric $2$-designs which maximise certain parameters. Their constructions depend on certain quadratic polynomials with integer coefficients taking prime power values. Many of their polynomials satisfy three simple necessary conditions which Bunyakovsky~\cite{Bun-1857} in 1857 conjectured were also sufficient for any polynomial to take infinitely many prime values. Unfortunately, this conjecture has been proved only for polynomials of degree~$1$ (Dirichlet's Theorem on primes in an arithmetic progression). Nevertheless, the Bateman--Horn Conjecture~\cite{BH}, dating from 1962 and also proved only for degree $1$, gives estimates $E(x)$ for the number $Q(x)$ of positive integers $t\le x$ at which a given polynomial takes prime values. Using a recent improvement to the Bateman--Horn Conjecture due to Li~\cite{W.Li}, we calculated these estimates $E(x)$ for some of the simpler polynomials arising in~\cite{ADP}, taking $x=10^{8}$, and compared them with the actual numbers $Q(x)$ found by computer searches. As in various other applications of this conjecture (see~\cite{JZ21, JZ21a} for example), the estimates $E(x)$ are remarkably close to the actual values $Q(x)$. Although this  does not prove the existence of infinite families of block designs, the accuracy of the estimates, together with the abundance of examples found, provides strong evidence for it, and it also adds to the growing body of evidence in favour of the more general Bunyakovsky and Bateman--Horn Conjectures.

There have been many number-theoretic applications of the Bateman--Horn Conjecture (see~\cite{AFG} for a survey), and a handful in areas such as combinatorics~\cite{EKN}, cryptography~\cite{BG, Scholl17, Scholl19, Sha}, elliptic curves~\cite{BPS, DS}, error-correcting codes~\cite{Kim} and fast integer multiplication~\cite{CT}. It seems likely that the present paper and~\cite{ADP} represent its first application to block designs, just as~\cite{JZ21, JZ21a} are the first in the areas of dessins d'enfants and permutation groups.


\section{Primes versus prime powers}\label{sec:versus}

Although the problem in~\cite{ADP} requires prime power values of certain polynomials $f_{n,r}(t)\in{\mathbb Z}[t]$, it is easier to estimate the distribution of their prime values, using the Prime Number Theorem and conjectures based on it. This restriction is no great loss, as the vast majority of prime powers, up to any given large bound, are in fact prime: if $\pi(x)$ is the usual function counting primes \linebreak $p\le x$, and $\Pi(x)$ is its analogue for prime powers $p^e\le x$, then $\pi(x)/\Pi(x)\to 1$ (quite rapidly) as $x\to\infty$. For example, $\pi(10^6)/\Pi(10^6)=78\,498/78\,734=0.9970002\ldots$, while  $\pi(10^9)/\Pi(10^9)=50\,847\,534/50\,851\,223=0.999927\ldots$ (see~\cite{Cook}). 
Nevertheless, we carried out a more restricted search, 
over $t=1,\ldots,10^7$, for prime power values $f_{n,r}(t)$ of the chosen polynomials, 
finding just a few squares and one cube (see Section~\ref{sec:powers}). However, in Section~\ref{t=0} we show how to realise any even power $p^{2i}>9$ of an odd prime $p$ as $f_{n,r}(0)$ for some polynomial $f_{n,r}$, a situation which has some interest for the construction of block designs.


\section {The Bunyakovsky Conjecture}\label{sec:Bunyakovsky}

If a non-costant polynomial $f(t)\in{\mathbb Z}[t]$ is to take infinitely many prime values for $t\in{\mathbb N}$ (equivalently, if it is prime for infinitely many such $t$), then the following conditions must be satisfied: 
\begin{itemize}
\item[(a)] $f$ must have a positive leading coefficient (otherwise it will take only finitely many positive values);
\item[(b)] $f$ must be irreducible in ${\mathbb Z}[t]$ (otherwise all but finitely many of its values will be composite);
\item[(c)] $f$ must not be identically zero modulo any prime $p$ (otherwise all its values will be divisible by $p$).
\end{itemize}

In 1857 Bunyakovsky~\cite{Bun-1857} conjectured that these three necessary conditions are also sufficient. (Condition~(c) is needed to avoid examples such as $t^2+t+2$, which satisfies (a) and (b) but takes only even values.) For instance, if this were true it would imply Landau's conjecture (studied also by Euler~\cite{Eul}) that there are infinitely many primes of the form $t^2+1$. However, the Bunyakovsky Conjecture has been proved only in the case where $f$ has degree $1$: this is Dirichlet's Theorem, that if $a$ and $b$ are coprime integers then there are infinitely many primes of the form $at+b$ (see~\cite[\S 5.3.2]{BS} for a proof).


\section{The Bateman--Horn Conjecture}\label{sec:BHC}

In 1962 Bateman and Horn~\cite{BH} proposed a very general conjecture (in what follows
we will use the abbreviation BHC) which comprises many previous conjectures and theorems 
and gives quantified versions of them. It deals with a finite set of polynomials simultaneously 
taking prime values. Though for our purposes it is sufficient to consider the case 
of a single polynomial, we give here the full version of the BHC. If we incorporate a 
recent improvement due to Li~\cite{W.Li}, we get the following statement:

\begin{conj}[Bateman and Horn, 1962; Li, 2019]
Let $f_1,\ldots,f_k\in\mathbb{Z}[t]$ be coprime polynomials satisfying conditions
(a) and (b) of the Bunyakovsky Conjecture, and let their product $f=f_1\cdots f_k$ 
satisfy condition (c). Denote by $Q(x)$ the number of $t\in\mathbb{N}$, $t\le x$, 
such that all $f_i(t)$, $i=1,\ldots,k$, are prime. Then the asymptotic estimate 
$E(x)$ for the number $Q(x)$ is given by the following formula:
\begin{equation}\label{eq:BH-Q}
Q(x)\sim E(x):=C\negthinspace\int_a^x\negthinspace\frac{dt}{\prod_{i=1}^k\ln f_i(t)}
\quad\hbox{as}\quad x\to\infty
\end{equation}
where
\begin{equation}\label{eq:BH-C}
C=C(f):=\prod_p\left(1-\frac{1}{p}\right)^{-k}\left(1-\frac{\omega_f(p)}{p}\right)
\end{equation}
with the product over all primes $p$, and where $\omega_f(p)$ is the number of congruence 
classes $t\in{\mathbb Z}_p$ such that $f(t)=0$. In (\ref{eq:BH-Q}), one chooses $a\ge 2$ 
large enough that the range of integration avoids singularities, where some $f_i(t)=1$. 
(In our applications we can always take $a=2$.)
\end{conj}

\begin{lemm}[Constant $C(f)$]
The product in (\ref{eq:BH-C}) converges to a constant $C>0$.
\end{lemm}

This statement is far from trivial. Bateman and Horn, in their original paper
\cite{BH}, limit themselves to a few hints. The first detailed proof was recently 
published in \cite[Theorem 5.4.3]{AFG}, and it takes seven pages. 

Since the integral in (\ref{eq:BH-Q}) diverges, we get the following

\begin{coro}[Infinitely many prime values]
The estimate $E(x)\to\infty$ as $x\to\infty$; therefore, $Q(x)$ also goes to infinity:
there are infinitely many integers $t\in\mathbb{N}$ such that all $f_i(t)$, $i=1,\ldots,k$,
are simultaneously prime.
\end{coro}

As in the case of the Bunyakovsky Conjecture, the BHC, even when restricted to a single polynomial $f$, has been proved only in the case where $\deg f=1$. This is the quantified version of Dirichlet's Theorem, that for fixed coprime $a$ and $b$ the number of $t\le x$ such that $at+b$ is prime is asymptotic to
\[\frac{1}{\varphi(a)}\int_2^x\frac{dt}{\ln(at+b)},\]
where $\varphi$ is Euler's totient function. (Equivalently, the primes in the arithmetic progression $at+b$ are asymptotically equally distributed among the $\varphi(a)$ congruence classes of units mod~$a$; see \linebreak \cite[\S5.3.2]{BS} for a proof.)

An earlier special case of the BHC, applicable to a single quadratic polynomial $f$, is the Conjecture F of Hardy and Littlewood~\cite{HL}, giving similar estimates. For this reason, the constants $C(f)$ are sometimes known as Hardy--Littlewood constants.


\section{Heuristic argument for the ingredients of the Bateman--Horn Conjecture}\label{sec:heuristic}

Here we give a heuristic argument to explain certain ingredients 
of the formula (\ref{eq:BH-Q}) for the Bateman--Horn estimate $E(x)$. 

The Prime Number Theorem provides two asymptotic estimates for the number $\pi(x)$ of primes $p\le x$ as $x\to\infty$, namely
\begin{equation}\label{eq:PNT}
\pi(x)\sim\frac{x}{\ln x} \quad\quad\hbox{and}\quad\quad \pi(x)\sim {\rm Li}(x):=\int_2^x\frac{dt}{\ln t}.
\end{equation}
The first is easy to use, but not very accurate; the second, involving the {\sl offset 
logarithmic integral function}\/ ${\rm Li}(x)$, is harder to use but much more accurate. 
For example, the number of primes up to $10^{25}$ was computed in 2013 by J. Buethe, 
J. Franke, A. Jost, and T. Kleinjung (see \cite{OEIS-A006880}): it is equal to
$\pi(10^{25})=176\,846\,309\,399\,143\,769\,411\,680$.
The formula $x/\ln x$ approximates this number with the relative error $-1.77\%$, 
while the relative error of the estimate ${\rm Li}(10^{25})$ is $3.12\cdot 10^{-11}\%$.

In either case, (\ref{eq:PNT}) suggests that one can regard $1/\ln x$ as the probability 
that $x$ (or, rather, a randomly-chosen number close to $x$) is prime.
Consider the ``random variables'' $\xi_i(t)=1$ if $f_i(t)$ is prime, and $\xi_i(t)=0$ 
otherwise. The ``probability'' that $\xi_i(t)=1$ is $1/\ln f_i(t)$. If, in addition, we
presume that these variables, for any given $t$, are independent, then the probability
that all $f_i(t)$ are prime, or, in other words, the probability that the product
$\eta(t):=\xi_1(t)\cdots\xi_k(t)$ is equal to $1$, is
$\displaystyle P(t)=\frac{1}{\prod_{i=1}^k\ln f_i(t)}$. Notice that the mean value
of $\eta(t)$ is the expected value ${\rm E}(\eta(t))=P(t)$.

The random variable $\eta(t)$ is a ``counting function'': as a first estimate for 
the number of $t\le x$ such that all $f_i(t)$ are prime, we may take the average 
number of times this variable is equal to 1. Let us choose $a$ so that all 
$f_i(t)>1$ for $t\ge a$. Then, as $t$ goes from $a$ to $x$, we have
\begin{equation}\label{eq:naive}
{\rm E}\left(\sum_{t=a}^x\eta(t)\right) = \sum_{t=a}^x {\rm E}(\eta(t)) = 
\sum_{t=a}^x P(t) \approx \int_a^x P(t)dt = \int_a^x\frac{dt}{\prod_{i=1}^k \ln f_i(t)}.
\end{equation}

We cannot present any profound reasons for considering $f_1(t),\ldots,f_k(t)$ as 
independent for any given $t$, but at least this assumption stands the test of a 
great number of experiments. However, the same is not true when we vary the variable $t$.
Therefore, a correcting term may be needed, and this is the constant $C(f)$.

First, if $f(t_0)\equiv 0$ mod~$(p)$ for some integer $t_0$ and prime $p$, then 
$f(t)\equiv 0$ mod~$(p)$ for all $t\equiv t_0$ mod~$(p)$. We would like to avoid
the situation when $f(t)$ is divisible by $p$ (or, equivalently, at least one of
$f_i(t)$ is divisible by $p$). The ``probability'' of the opposite event is 
$a_p=1-\frac{\omega_f(p)}{p}$.

Second, the probability that a ``randomly chosen'' $k$-tuple of integers (whatever that
means) does not contain any element divisible by $p$ is 
$b_p=\left(1-\frac{1}{p}\right)^k$. The ratio $a_p/b_p$ used in the
product (\ref{eq:BH-C}) resembles the conditional probability, though it is not one
since it may well be $>1$.

What remains is to assemble different parts of this Lego, but the corresponding procedure
will need a long discussion and a self-coherent construction of a ``probabilistic model''
of what takes place, so we stop here. Anyway, we are not supposed to give a proof of the
BHC; we only provide some plausible speculations on the matter. ``The proof of the pudding 
is in its eating'': the conjecture works well, even surprisingly well, and this is what
is important about it.

The BHC is a statement involving a $k$-tuple of polynomials $f_1,\ldots,f_k$. In what follows
we will work with a single polynomial, so that $k=1$, and $f_1$ will from now on be denoted
by $f$. There is no contradiction with the previous notation where $f$ denoted the product
of the polynomials $f_i$.


\section{The constant $C(f)$}\label{sec:const_C}

Computing the constant $C(f)$ is a challenging problem in itself. As already mentioned
above, the mere existence of a limit is a non-trivial fact. By the way, the convergence
is not absolute: by changing the order of factors we may get a different limit value.
This is perhaps one of the manifestations of the fact that the ``probabilistic measure" on
$\mathbb{N}$ represented by the density is finitely additive but not countably additive.
To make matters worse, the rate of convergence is, as one of our colleagues has put it,
``frustratingly slow''.

In Section \ref{sec:omega} we discuss the computation of $\omega_f(p)$ for a single
quadratic polynomial $f$. We will see that it involves rather subtle number-theoretic
methods, mainly the quadratic reciprocity law. The case of cubic polynomials is
treated in \cite{ShL}.

A highly advanced method, though still for a single polynomial, was proposed by 
H.~Cohen~\cite{Coh}. For a quadratic polynomial it involves the techniques 
of $L$-functions and, in particular, of the Riemann $\zeta$-function. For polynomials 
of degree greater than 2 one also needs to know the Galois group of the polynomial
in question as well as the irreducible representations of this group. An intermediate 
way is chosen by Li \cite{W.Li} (with a reference to K.~Conrad): he also uses 
$L$-functions but in a simpler way than in~\cite{Coh}.

We have no intention to compete with the above specialists. Therefore, we have computed 
the products only over primes $p\le 10^8$. The constants $C(f)$ thus obtained already
give excellent results in approximating the numbers of prime values of the polynomials 
we study in this paper.

Another interesting question is, how large (or how small)
the constant $C(f)$ could be. For example, for the well-known Euler polynomial
$t^2+t+41$ (taking prime values for $t=0,1,2,\ldots,39$) the constant found in
\cite{Coh} is $3.31977318$, while for the polynomial $t^2+t+75$ it is $0.31097668$
(in fact, the author gives 39 correct digits for both numbers). Since the integrals
$\displaystyle \int_2^x\frac{dt}{\ln f(t)}$ for these two polynomials 
are very close to each other for large $x$, we conclude that the first polynomial produces, 
approximately, $10.7$ times as many primes as the second one.

In \cite{AFG}, an example of a polynomial $f(t)=t^2+t+a$ is given, with the constant
$5.4972\ldots$. In this example, the coefficient $a$ is a 219-digit number. This number
is not prime, and the discriminant is not prime either. We did not try to factorize
them.

A more systematic search was carried out in \cite{JW} (which contains many interesting 
examples) and \cite{Rivin} (which is, mostly, an experimental work). The champion, 
found by Rivin, among the {\em monic quadratic}\/ polynomials is $t^2-2619t+1291$, 
with the corresponding constant equal to $6.3722$. The discriminant of this polynomial 
is $6\,853\,997$, which is a prime. 
The author has also found a monic polynomial of degree 6 with a yet greater constant $C(f)$; 
the discriminant of the polynomial in question is $53\times\mbox{(a prime with 37 digits)}$. 
The data based on a large random sample of monic polynomials $f$ with their coefficients 
bounded by a large constant $N$ suggest that $C(f)$ obeys the log-normal distribution. 
If this observation turns out to be true then, in general, $C(f)$ is unbounded. Other 
plausible observations based on the experimental data are as follows: (a) the mean 
value of $C(f)$ is 1; (b) for monic polynomials whose coefficients (other than the 
leading one) are bounded by $N$, the maximum value of $C(f)$ grows like 
$C(f)=O(\log\log N)$. At the same time, in \cite{JW} it is conjectured that among the
polynomials of the type $f(t)=t^2+t-a$ the maximum value of $C(f)$ is equal to
$5.65726388$, and it is attained for the 71-digit number
{\small
\[
a = 33\,251\,810\,980\,696\,878\,103\,150\,085\,257\,129\,508\,857\,312\,847\,751\,498\,190\,
{349\,983\,874\,538\,507\,313.}
\]}
\indent Finally, an example of a non-monic polynomial (also from \cite{Rivin}): taking $f(t)=At^2+1$ 
with $A=\prod_{i=1}^{30}p_i$ the product of the first 30 primes from $p_1=2$ to
$p_{30}=113$, we get a rather large constant $C(f)\approx 9.5$. There is no mystery: 
all the factors $(1-1/p)^{-1}$ in (\ref{eq:BH-C}) are greater than~1, while 
$\omega_f(p)=0$ for $p=2,3,\ldots,113$ (in fact, also for $p=127$ but not for 131), 
making the factors $(1-\omega_f(p)/p)$ equal to~1, so that they do not compensate for 
a steadily growing product which reaches, at this initial stage, approximately $8.78$.


\section{Block designs}\label{sec:designs}

Here, in order to provide motivation for our particular choice of polynomials $f$, we briefly summarise the construction in~\cite{ADP} of block designs requiring certain polynomials to take prime power values. Readers who are interested only in the number-theoretic aspects of this problem can safely omit this section.

A 2-$(v, k, \lambda)$ {\em design\/} $\mathcal D$ consists of a set $\mathcal P$ of $v$ points, together with a set $\mathcal B$ of $k$-element subsets of $\mathcal P$ called blocks, such that each pair of points lie in exactly $\lambda$ blocks. (This implies that each point lies in the same number of blocks.) The {\em automorphisms\/} of $\mathcal D$ are the permutations of $\mathcal P$ which leave the set $\mathcal B$ invariant; they form a group ${\rm Aut}\,{\mathcal D}$.

If a subgroup $G\le{\rm Aut}\,{\mathcal D}$ acts transitively on blocks then it also acts transitively on points. The latter action could be imprimitive, leaving invariant a partition $\mathcal C$ of $\mathcal P$ with $d\ge 2$ classes, each of size $c\ge 2$, so that $cd=v$. Delandtsheer and Doyen showed in~\cite{DD} that in this case there exist positive integers $m$ and $n$ such that
\[mc+n=\binom{k}{2}=nd+m.\]
These integers $m$ and $n$ are the Delandtsheer-Doyen parameters of $\mathcal D$, with $n$ and $mc$ the numbers of unordered pairs of points in any given block, lying in the same or in different classes of $\mathcal C$.

In~\cite{ADP}, Amara, Devillers and Praeger have explored the restrictions these parameters place on subgroups $G$ of ${\rm Aut}\,\mathcal D$. Let $K$ denote the permutation group of degree $d$ induced by $G$ on the set of classes in $\mathcal C$, and let $H$ be the permutation group of degree $c$ induced on any class in $\mathcal C$ by its setwise stabiliser in $G$, so that $G$ is embedded in the wreath product $H\wr K\le {\rm S}_c\wr{\rm S}_d$. The {\sl rank\/} ${\rm Rank}(X)$ of any transitive permutation group $X$ on a set $\Omega$ is the number of orbits of a point-stabiliser $X_{\alpha}\;(\alpha\in\Omega)$, or equivalently of $X$ on $\Omega\times\Omega$; similarly, the {\em pair-rank\/} ${\rm PairRank}(X)$ is the number of orbits of $X$ on unordered pairs of distinct elements of $X$, so that $({\rm Rank}(X)-1)/2\le{\rm PairRank}(X)\le{\rm Rank}(X)-1$. The main result of~\cite{ADP} is that in the above circumstances
\[\frac{{\rm Rank}(H)-1}{2}\le{\rm PairRank}(H)\le n \quad\hbox{and}\quad \frac{{\rm Rank}(K)-1}{2}\le{\rm PairRank}(K)\le m.\] 

The authors of~\cite{ADP} give several constructions of designs $\mathcal D$ in which the ranks and pair-ranks of $H$ and $K$ attain these upper bounds. One construction requires {\em useful pairs} of integers $n, c$ with the properties that $n\ge 2$ and $c$ is a prime power such that
\[c\equiv 1\, {\rm mod}\,(2n)\quad\hbox{and}\quad c+n=\binom{k}{2}\;\,
\hbox{for some integer}\;\, k\ge 2n.\]
They need $c$ to a prime power in order to define $H$ to be the unique subgroup of index $n$ in ${\rm AGL}_1(c)$, acting naturally on the field ${\mathbb F}_c$, while they take $K={\rm S}_d$ acting naturally on ${\mathbb Z}_d$, so that $G:=H\wr K$ has a transitive but imprimitive induced action on ${\mathcal P}={\mathbb F}_c\times{\mathbb Z}_d$ with $d$ classes of size $c$. By taking $d=1+\frac{c-1}{n}$ (the number of orbits of $H$ on ${\mathbb F}_c$) and defining ${\mathcal B}$ to be the set of images under $G$ of a carefully-chosen $k$-element subset $B\subset{\mathcal P}$ they obtain a 2-$(cd, k, \lambda)$ design $\mathcal D$ for some $\lambda$, admitting $G$ as a block-transitive and point-imprimitive group of automorphisms. This design has Delandtsheer--Doyen parameters $m=1$ and $n$, with ${\rm Rank}(H)={\rm PairRank}(H)+1=n+1$ and ${\rm Rank}(K)={\rm PairRank}(K)+1=2$.

The conditions for the pair $n, c$ to be useful imply that, if $r$ denotes the least positive remainder of $k$ mod~$(4n)$, then $\binom{r}{2}\equiv\binom{k}{2}\equiv n+1$ mod~$(2n)$. Thus, for fixed positive integers $n\ge 2$ and $r<4n$ with $\binom{r}{2}\equiv n+1$ mod~$(2n)$ they need integers $k=4nt+r$ for some integer $t\ge 0$ such that
\[f_{n,r}(t):=\binom{k}{2}-n=8n^2t^2+2n(2r-1)t+\left(\frac{r(r-1)}{2}-n\right)\]
is a prime power $c$. If the polynomial $f_{n,r}$ takes prime power values for infinitely many integers $t\ge 0$ then this construction yields an infinite family of block designs with the required parameters and symmetry properties.

\begin{exam} 
The smallest useful pair $(n,c)$ is $(2,13)$, with $r=k=6$ and $d=7$, so that the corresponding design $\mathcal D$ has $cd=91$ points and $|G|=78^7\cdot 7!$ automorphisms. This example arises from the polynomial $f_{n,r}(t)=f_{2,6}(t)=32t^2+44t+13$ taking the value $c=13$ at $t=0$. Note that $f_{2,6}(1)=89$ is prime, giving a design on $cd=89\cdot 45=4005$ points, whereas $f_{2,6}(2)=697=17\cdot 41$ is not a prime power and therefore does not correspond to a design in this family.

The smaller polynomial $f_{2,3}(t)=32t^2+20t+1$ has its first prime power value $f_{2,3}(1)=53$, giving a design on $53\cdot27=1431$ points. 
\end{exam}

Note that, although this construction of block designs applies to any integer $t\ge 0$ such that $f_{n,r}(t)$ is a prime power, the number-theoretic conjectures and estimates we use are stated in terms of integers $t\ge 1$. This is not a problem here, since we are not concerned with individual block designs but with the existence or otherwise of infinite families of them. In any case, the value $f_{n,r}(0)=\frac{r(r-1)}{2}-n$ is easily dealt with (see Section~\ref{t=0}).


\section{Verifying the Bunyakovsky conditions}\label{sec:verifying}

The polynomials $f$ of interest in~\cite{ADP}, and hence the main focus of this note, are those of the form
\begin{equation}\label{eq:f}
f(t)=f_{n,r}(t)=8n^2t^2+2n(2r-1)t+\left(\frac{r(r-1)}{2}-n\right)
\end{equation}
for integers $n\ge 2$ and $r\ge 1$ with
\begin{equation}\label{eq:conditions}
r<4n\quad \hbox{and} \quad \frac{r(r-1)}{2}\equiv n+1\; \hbox{mod}\,(2n).
\end{equation}
Note that this last condition implies that $r\ge 3$.

\begin{lemm}\label{le:abc}
If a polynomial $f=f_{n,r}$ of the form~(\ref{eq:f}) satisfies~(\ref{eq:conditions}), it also satisfies Bunyakovsky's conditions~(a) and (c); it satisfies his condition~(b) if and only if $n$ is not a triangular number $a(a+1)/2$, $a\in{\mathbb N}$.
\end{lemm}

\begin{proof} Clearly $f$ satisfies condition~(a) since $n\ge 1$. As a quadratic polynomial, $f$ is reducible over~$\mathbb Z$ if and only if its discriminant $\Delta$ is a perfect square. Here
\begin{equation}\label{eq:Delta}
\Delta=4n^2(2r-1)^2-32n^2\left(\frac{r(r-1)}{2}-n\right)=4n^2(8n+1),
\end{equation}
and this is a square if and only if $8n+1$ is. Simple algebra shows that the solutions $n\in{\mathbb N}$ of $8n+1=l^2\;(l\in{\mathbb Z})$ are the triangular numbers $n=1, 3, 6, 10,\ldots$, those of the form $a(a+1)/2$ for some $a=(l-1)/2\in{\mathbb N}$ (readers may enjoy finding a geometric `proof without words' for this), so $f$ will satisfy~(b) if and only if $n$ does not have this form. 

We now check condition~(c). If a prime $p$ divides $2n$ then $f$ reduces mod~$(p)$ to  a constant polynomial; this takes the value $1$ since $r(r-1)/2\equiv n+1$ mod~$(2n)$, so $f$ is not identically zero mod~$(p)$. If $p$ does not divide $2n$ then $f$ reduces to a quadratic polynomial, with at most two roots, so again it cannot be identically zero.
\end{proof}

In order to apply the Bateman--Horn Conjecture to the polynomials $f_{n,r}$, we therefore restrict attention to those for which $n$ is not a triangular number.


\section{Calculating $\omega_f(p)$ for $f_{n,r}$}\label{sec:omega}

Recall that $\omega_f(p)$, which appears in the infinite product (\ref{eq:BH-C}), is the number of roots of $f$ mod~$(p)$ for each prime $p$. We saw in the proof of Lemma~\ref{le:abc} that $\omega_f(p)=0$ for any prime $p$ dividing $2n$. Primes $p$ dividing $8n+1$ (and thus not dividing $2n$) give $\Delta\equiv 0$ mod~$(p)$ by (\ref{eq:Delta}), and hence $\omega_f(p)=1$ by the quadratic formula. Similarly, all other primes $p$ give $\omega_f(p)=2$ or $0$ as $8n+1$ is or is not a quadratic residue (non-zero square) mod~$(p)$.

In general, given any prime $p$ and integer $q$, one can determine whether or not $q$ is a quadratic residue mod~$(p)$ by using the Legendre symbol
\[\left(\frac{q}{p}\right)=
\begin{cases}
0\quad\; \hbox{if $q\equiv 0$ mod~$(p)$;}\\
1\quad\; \hbox{if $q$ is a quadratic residue mod~$(p)$;}\\
-1 \;\; \hbox{otherwise.}
\end{cases}.\]
(See~\cite[Chapter~7]{JJ} for quadratic residues and the Legendre symbol.) Clearly
\[\left(\frac{q}{p}\right)=\left(\frac{q'}{p}\right)\quad \hbox{if $q\equiv q'$ mod~$(p)$,}\]
and since the quadratic residues form a subgroup of index~$2$ in the group of units mod~$(p)$ we have the multiplicative property, that
\[\left(\frac{qq'}{p}\right)=\left(\frac{q}{p}\right)\left(\frac{q'}{p}\right)\]
for all $q, q'\in{\mathbb Z}$. Using these rules one can reduce the calculation of the Legendre symbol to the cases where $q$ is an odd prime. In such cases one can use the Law of Quadratic Reciprocity, that if $p$ and $q$ are distinct odd primes then
\[\left(\frac{q}{p}\right)=\left(\frac{p}{q}\right)\quad \hbox{if $p$ or $q\equiv 1$ mod~$(4)$,}\]
while
\[\left(\frac{q}{p}\right)=-\left(\frac{p}{q}\right)\quad \hbox{if $p\equiv q\equiv -1$ mod~$(4)$.}\]
We also have
\[\left(\frac{2}{p}\right)=1\;\hbox{or}\;-1\quad\hbox{as}\quad p\equiv\pm 1\;\hbox{or}\;\pm 3\;{\rm mod}\,(8),\]
and
\[\left(\frac{-1}{p}\right)=1\;\hbox{or}\;-1\quad\hbox{as}\quad p\equiv 1\;\hbox{or}\;-1\;{\rm mod}\,(4).\]
By iterating these rules one can reduce the values of $p$ and $q$ until they are small enough to be dealt with by inspection.

We have seen that if $f=f_{n,r}$ then $\omega_f(p)=0$ for all primes $p$ dividing $2n$. For primes $p$ not dividing $2n$, by the definitions of the function $\omega_f$ and the Legendre symbol, the quadratic formula gives $\omega_f(p)=\left(\frac{\Delta}{p}\right)+1$. We will use this in the following examples.

\begin{exam}\label{ex:n=2}
The smallest value of $n$ which is not a triangular number is $n=2$, giving $8n+1=17$. Since $17\equiv 1$ mod~$(4)$ we have
\[\left(\frac{17}{p}\right)=\left(\frac{p}{17}\right)\]
for any odd prime $p$. By squaring integers one sees that the quadratic residues mod~$(17)$ are $\pm 1,\pm 2, \pm 4$ and $\pm 8$ (in fact, under multiplication mod~$(17)$ they form a cyclic group of order $8$, generated by $2$). Thus, if $f=f_{2,r}$ for some $r$ then $\omega_f(p)=2$ for odd primes $p\equiv\pm 1,\pm 2, \pm 4$ or $\pm 8$ mod~$(17)$, while $\omega_f(p)=0$ for primes $p\equiv\pm 3,\pm 5, \pm 6$ or $\pm 7$ mod~$(17)$. For the remaining primes $p$ we have $\omega_f(2)=0$ and $\omega_f(17)=1$.
\end{exam}

\begin{exam}\label{ex:n=4}
The second smallest value of $n$ which is not a triangular number is $n=4$, giving $8n+1=33$. In this case multiplicativity and quadratic reciprocity give
\[\left(\frac{33}{p}\right)=\left(\frac{3}{p}\right)\left(\frac{11}{p}\right)=\left(\frac{p}{3}\right)\left(\frac{p}{11}\right)\]
for all odd primes $p\ne 3, 11$, since $3\equiv 11$ mod~$(4)$ so that any minus signs cancel. Now the quadratic residues mod~$(3)$ and mod~$(11)$ are $1$ and $1, 3, 4, 5, 9$ respectively.
The primes $p$ for which $33$ is a quadratic residue mod~$(p)$ are those which are both residues or both non-residues mod~$(3)$ and mod~$(11)$, so solving the relevant pairs of simultaneous congruences gives the classes $\pm1, \pm 2, \pm 4, \pm 8, \pm 16$ mod~$(33)$ (forming a cyclic group generated by $2$). If $f=f_{4,r}$ for some $r$ then for odd primes $p$ in these classes we have $\omega_f(p)=2$, whereas for $p\equiv\pm 5, \pm 7, \pm 10, \pm 13, \pm 14$ mod~$(33)$ we have $\omega_f(p)=0$. For the remaining primes $p$ we have $\omega_f(2)=0$ and $\omega_f(3)=\omega_f(11)=1$.

Notice that the classes $\pm 3,\pm 6,\pm 9,\pm 12,\pm 15$ and $\pm 11$ are 
not present in the above two lists: they are not coprime with~$33$ and therefore cannot 
be residues of a prime $p>11$ modulo $33$.
\end{exam}

\begin{exam}\label{ex:n=5}
The next case $n=5$ is similar to Example~\ref{ex:n=2} since $8n+1=41$ is prime. 
We find that $\omega_f(2)=\omega_f(5)=0$ and $\omega_f(41)=1$, while for other primes $p$ we have $\omega_f(p)=2$ or $0$ as $p$ is or is not a quadratic residue mod $(41)$. These are
$\pm 1, \pm 2, \pm 4, \pm 5, \pm 8, \pm 9, \pm 10, \pm 16, \pm 18, \pm 20$.
\end{exam}

For other permitted values of $n$ the process is similar: thus for $n=7$, $8$ and $9$ we have $8n+1=57=3\cdot 19$, $65=5\cdot 13$ and $73$ which is prime. However, in some cases the process can be lengthier, depending on the factorisation of $8n+1$. We give just one more typical example.

\begin{exam}\label{ex:n=13}
If $n=13$ then $8n+1=105=3\cdot 5\cdot 7$. Since $3\equiv 7\equiv -1$ mod~$(4)$ while $5\equiv 1$ mod~$(4)$ we have
\[\left(\frac{105}{p}\right)=\left(\frac{3}{p}\right)\left(\frac{5}{p}\right)\left(\frac{7}{p}\right)
=\left(\frac{p}{3}\right)\left(\frac{p}{5}\right)\left(\frac{p}{7}\right)\]
for all primes $p\ge 11$, so for such $p$ we have $\omega_f(p)=0$ or $2$ as $p$ is a quadratic residue modulo an even or odd number of the primes $3, 5$ and $7$. Since the quadratic residues modulo these primes are generated by $1$, $-1$ and $2$ respectively, this is easily determined in terms of congruences mod~$(105)$. (For some primes $p$, short-cuts are possible: for instance $105\equiv -1$ mod~$(53)$, so $\left(\frac{105}{53}\right)=\left(\frac{-1}{53}\right)=1$, and similarly $\left(\frac{105}{107}\right)=\left(\frac{-1}{107}\right)\left(\frac{2}{107}\right)=(-1)^2=1$.) For $p=3$, $5$ or $7$ we have $\omega_f(p)=1$, while $\omega_f(2)=\omega_f(13)=0$.
\end{exam}

Since the values of $\omega_f(p)$ depend only on a few simple congruences for $p$, it is straightforward to program Maple to determine the factors in the infinite product (\ref{eq:BH-C}) and hence to evaluate $C$. Note also that this part of the process depends only on $n$, and not on $r$, so that polynomials $f_{n,r}$ with the same parameter $n$ can be dealt with simultaneously.


\section{Evaluating the estimates}\label{sec:evaluating}

Since the factors in (\ref{eq:BH-C}) approach $1$ quite slowly as $p\to\infty$, 
convergence of this infinite product is rather slow, and one needs to multiply 
many terms in order to obtain good approximations for $C$. In our computations 
we used all the primes $p\le10^8$.

Maple calculates the definite integral in (\ref{eq:BH-Q}) by numerical quadrature. We found that running times were less than a second. Bateman and Horn simplified this part of the process by replacing $\ln(f(t))$ with $\deg(f)\ln(t)$ in (\ref{eq:BH-Q}), thus ignoring the leading coefficient of $f$ together with all non-leading terms. No doubt, working in the early 1960s without resources such as Maple, they found that this shortcut was essential, especially in cases involving more than one polynomial. Li's recent improvement~\cite{W.Li}, using $\ln(f(t))$, certainly leads to more accurate estimates. In fact, the non-leading terms have remarkably little effect on the value of the integral (so again, $r$ is almost irrelevant), whereas most of the extra accuracy comes from including the leading coefficient. For instance, the estimates $E(10^8)$ for the polynomials $f_{2,3}(t)=32t^2+20t+1$ and $f_{2,6}(t)=32t^2+44t+13$, given in the next section, differ by only $0.29$.

For each polynomial $f$ we used Maple to find the actual number $Q(x)$ of prime values of $f(t)$ for $t\le x=10^8$. Since, for example, $f_{5,r}(10^8)\approx 2\cdot10^{18}$, this was the most time-consuming part of our computations, with running times of about two hours on a modest laptop. Maple uses the Rabin--Miller primality test, which is probabilistic rather than deterministic. If an integer is prime, the test will always declare it to be prime. If an integer is composite, the test may incorrectly declare it to be prime, but the probability of this happening is so small that in forty years of use of the test, no such incident has ever been reported. In our case, we found so many prime values of the polynomials $f$ which we considered that, even if we have been very unlucky and a few of them are actually composite, this will have a negligible effect on our evidence. 


\section{The estimates and their accuracy}\label{sec:estimates}

\subsection{The case $n=2$.} The smallest allowed value for the parameter $n$ is $2$, so condition~(\ref{eq:conditions}) implies that $r=3$ or $6$. 
In either case, evaluating $\omega_f(p)$ as in Example~\ref{ex:n=2}, and taking the product 
in (\ref{eq:BH-C}) over all primes $p\le 10^8$, we found that $C=4.721240276\ldots$. Putting 
$r=3$ gives
\[f(t)=f_{2,3}(t)=32t^2+20t+1.\]
Taking $x=10^i$ for $i=3, 4, \ldots, 8$ we found the estimates $E(x)$ for the values of $Q(x)$ shown in Table~\ref{tab:n=2,r=3}. The final column, showing the relative error, reveals the accuracy of these estimates. 

\begin{table}[htbp]
\begin{center}
\begin{tabular}{c|c|c|c}
$x$	    & $Q(x)$  & $E(x)$		 	& relative error \\
\hline
$10^3$	& 326		 		& $314.49$			& $-3.53\%$ \\
$10^4$	& 2421		 		& $2404.86$			& $-0.67\%$ \\
$10^5$	& 19\,394			& $19\,438.26$		& $0.23\%$  \\
$10^6$	& 162\,877			& $163\,182.75$		& $0.19\%$  \\
$10^7$	& 1\,405\,448		& $1\,406\,630.14$	& $0.084\%$ \\
$10^8$ & 12\,357\,532		& $12\,362\,961.06$ & $0.044\%$
\end{tabular}
\end{center}
\caption{Numbers $Q(x)$ and estimates $E(x)$ for $f_{2,3}$.}
\label{tab:n=2,r=3}
\end{table}

\subsection{The cases with $n\le 9$}

\begin{table}[htbp]
\begin{center}
\begin{tabular}{c|c|c|c|c|c}
$(n,r)$  & $f_{(n,r)}(t)$    & $C(f)$    & $Q(10^8)$      & $E(10^8)$      & relative error \\
\hline\hline
$(2,3)$  & $32t^2+20t+1$     & $4.72124$ & $12\,357\,532$ & $12\,362\,961.06$ & $0.0439\%$  \\
$(2,6)$  & $32t^2+44t+13$    &           & $12\,363\,849$ & $12\,362\,960.77$ & $-0.0072\%$ \\
\hline
$(4,7)$  & $128t^2+104t+17$  & $3.20688$ & $8\,100\,174$  & $8\,102\,333.64$  & $0.0267\%$  \\
$(4,10)$ & $128t^2+152t+41$  &           & $8\,104\,531$  & $8\,102\,333.57$  & $-0.0271\%$ \\
\hline
$(5,4)$  & $200t^2+70t+1$    & $5.62398$ & $14\,052\,016$ & $14\,050\,339.22$ & $-0.012\%$  \\
$(5,9)$  & $200t^2+170t+31$  &           & $14\,049\,951$ & $14\,050\,339.05$ & $0.003\%$   \\
$(5,12)$ & $200t^2+230t+61$  &           & $14\,057\,558$ & $14\,050\,338.95$ & $-0.051\%$  \\
$(5,17)$ & $200t^2+330t+131$ &           & $14\,049\,868$ & $14\,050\,338.79$ & $0.003\%$   \\
\hline
$(7,9)$  & $392t^2+238t+29$  & $3.82010$ & $9\,381\,546$  & $9\,385\,428.26$  & $0.0415\%$  \\
$(7,13)$ & $392t^2+350t+71$  &           & $9\,387\,937$  & $9\,385\,428.21$  & $-0.0267\%$ \\
$(7,16)$ & $392t^2+434t+113$ &           & $9\,385\,853$  & $9\,385\,428.17$  & $-0.0045\%$ \\
$(7,20)$ & $392t^2+546t+183$ &           & $9\,387\,135$  & $9\,385\,428.11$  & $-0.0182\%$ \\
\hline
$(8,15)$ & $512t^2+464t+97$  & $3.22754$ & $7\,879\,429$  & $7\,877\,750.61$  & $-0.0213\%$ \\
$(8,18)$ & $512t^2+560t+145$ &           & $7\,879\,013$  & $7\,877\,750.57$  & $-0.0160\%$ \\
\hline
$(9,5)$  & $648t^2+162t+1$   & $5.41032$ & $13\,129\,138$ & $13\,129\,743.85$ & $0.0046\%$  \\
$(9,8)$  & $648t^2+270t+19$  &           & $13\,127\,661$ & $13\,129\,739.69$ & $0.0158\%$  \\
$(9,17)$ & $648t^2+594t+127$ &           & $13\,129\,080$ & $13\,129\,739.55$ & $0.0050\%$  \\
$(9,20)$ & $648t^2+702t+181$ &           & $13\,130\,890$ & $13\,129\,743.63$ & $-0.0087\%$ \\
$(9,29)$ & $648t^2+1026t+397$ &          & $13\,128\,036$ & $13\,129\,743.50$ & $0.0130\%$  \\
$(9,32)$ & $648t^2+1134t+487$ &          & $13\,128\,979$ & $13\,129\,743.46$ & $0.0058\%$
\end{tabular}
\end{center}
\smallskip
\caption{Complete list of irreducible polynomials $f_{n,r}$ defined in (\ref{eq:f}) and satisfying 
conditions (\ref{eq:conditions}), for $n\le 9$. The constants $C(f)$ are computed over 
primes $p\le 10^8$.}
\label{tab:up-to-n=9}
\end{table}

The process for the remaining polynomials $f_{n,r}$ with non-triangular numbers $n\le 9$ was similar, with $x=10^8$ in all cases. Table~\ref{tab:up-to-n=9} summarises the results.


\begin{rema}
The greater the leading coefficient of a polynomial, the more 
significant is Li's improvement in~\cite{W.Li} as compared with the initial 
Bateman--Horn formula in~\cite{BH}. For example, in the case of $f(t)=f_{9,5}(t)=648t^2+162t+1$, 
 we have two corresponding estimates
\[E_{\rm Li} = C\cdot\int_2^{10^8}\frac{dt}{\ln(f(t))}
\quad\hbox{and}\quad
E_{\rm BH} = \frac{C}{2}\cdot\int_2^{10^8}\frac{dt}{\ln(t)}\]
for $x=10^8$, with relative errors $0.0046\%$ and $18.7\%$ respectively.

This does not contradict the fact that the two estimates are asymptotically
equivalent. Indeed, the relative error of $E_{\rm BH}$ steadily decreases to 
approximately $2\%$ when the upper limit $x$ of the integration approaches $10^{70}$. 
(Of course, we did not count the true number $Q(x)$ of prime values of this
polynomial: instead, we took $E_{\rm Li}(x)$ as if it were the true 
value of $Q(x)$.)
\end{rema}


\section{Prime power values}\label{sec:powers}

We restricted our estimates to prime values of the polynomials $f_{n,r}$, since the Bunyakovsky and Bateman--Horn Conjectures have nothing to say about composite values. However, since the constructions of block designs in~\cite{ADP} apply to values which are prime powers, not just primes, we extended our computer searches to proper prime power values of some of these polynomials, for $t\le x=10^7$.

As predicted in Section~\ref{sec:versus}, we found very few proper prime power values, in comparison with the abundance of prime values. 
The values we found for $n\le 9$ and $t\le 10^7$ are shown in Table~\ref{tab:primepowers}. We observe that there is only one cube: all the other prime powers are squares.
The polynomials $f_{n,r}$ for the following pairs $(n,r)$ with non-triangular parameters $n\le 9$ gave no proper prime power values for $t\le 10^7$, so they have been omitted  from the table:
\[(2,6),\; (4,7),\; (4,10),\; (5,9),\; (5,12),\; (5,17),\; (7,9),\; (7,13),\; (7,16),\; (7,20),\]
\[(8,15),\; (8,18),\; (9,8),\; (9,20),\; (9,32). \]

\begin{table}[htbp]
\begin{center}
\begin{tabular}{c|c|c|c|c}
$(n,r)$ & polynomial $f_{n,r}$ & $t\le 10^7$ & $f_{n,r}(t)$ & power \\
\hline
$(2,3)$ & $32t^2+20t+1$ & 2 & 169 & $13^2$ \\
&& 8 & 2\,209 & $47^2$ \\
&& 78 & 196\,249 & $443^2$ \\
&& 282 & 2\,550\,409 & $1\,597^2$ \\
&& 9\,590 & 2\,943\,171\,001 & $54\,251^2$ \\
&& 23\,666 & 17\,923\,019\,113 & $2\,617^3$ \\
&& 90\,372 & 261\,348\,955\,729 & $511\,223^2$ \\
&& $3\,069\,998$ & $301\,596\,468\,440\,089$ & $17\,366\,533^2$ \\
\hline
$(5,4)$ & $200t^2+70t+1$ & 4 & 3\,481 & $59^2$ \\
&& 2\,044 & 835\,730\,281 & $28\,909^2$ \\
&& 4\,816 & 4\,639\,108\,321 & $68\,111^2$ \\
&& 163\,608 & 5\,353\,526\,985\,361 & $2\,313\,769^2$ \\
\hline
$(9,5)$ & $648t^2+162t+1$ & 3\,220 & 6\,719\,244\,841 & $81\,971^2$ \\
\hline
$(9,17)$ & $648t^2+594t+127$ & $1$ & $1\,369$ & $37^2$ \\
&& $49$ & $1\,585\,081$ & $1\,259^2$ \\
\hline
$(9,29)$ & $648t^2+1\,026t+397$ & 2 & 5\,041 & $71^2$ 
\end{tabular}
\end{center}
\caption{Proper prime power values for irreducible polynomials $f_{n,r}$ 
with $n\le 9$, $t\le 10^7$}
\label{tab:primepowers}
\end{table}


\section{Prime power values of reducible polynomials}\label{sec:reducible}

A reducible polynomial $f(t)=g(t)h(t)\in{\mathbb Z}[t]$ can take only finitely many prime values (with $g(t)$ or $h(t)$ equal to $\pm 1$), but could it take infinitely many prime power values? One way it might do so is if $g=h$ and this polynomial takes infinitely many prime values: Dirichlet's Theorem shows that this can happen with $\deg g=1$, and the Bunyakovsky Conjecture suggests that it can happen with $\deg g>1$. More generally, if $g$ is irreducible and takes infinitely many prime values $p$, then any power $f=g^e$ of $g$ takes infinitely many prime power values $p^e$. But what happens if $f$ has two or more distinct irreducible factors?

\begin{theo}\label{th:powers}
If $f$ is a polynomial in ${\mathbb Z}[t]$ with at least two different irreducible factors, then $f(t)$ is a prime power for only finitely many $t\in{\mathbb Z}$.
\end{theo}

\begin{proof} We first deal with a simple special case, and with $t\ge 0$. Suppose that $f=gh$ for distinct factors $g(t)=a_kt^k+\cdots$ and $h(t)=b_kt^k+\cdots$ in ${\mathbb Z}[t]$ of the same degree $k\ge 1$.
If there is some $t\in{\mathbb N}$ with $f(t)=p^e$ for a prime $p$ and integer $e\ge1$ then $g(t)=\pm p^i$ and $h(t)=\pm p^j$ for some integers $i, j\ge 0$ with $i+j=e$. If $i\ge j$ then
\[\frac{g(t)}{h(t)}=p^{i-j}\in{\mathbb Z}.\]
However, for all sufficiently large $t\in{\mathbb R}$ we have
\[\frac{g(t)}{h(t)}=\frac{a_kt^k+\cdots}{b_kt^k+\cdots}\to\frac{a_k}{b_k}\quad\hbox{strictly monotonically as}\;\; t\to+\infty,\]
so if there are infinitely many such $t\in{\mathbb N}$ with $i\ge j$ we have a sequence of integers $p^{i-j}$ converging strictly monotonically to $a_k/b_k$, which is impossible. A similar argument, with the factors $g$ and $h$ transposed, shows that there can be only finitely many such $t\in{\mathbb N}$ with $i<j$, so $f(t)$ is a prime power for only finitely many $t\in{\mathbb N}$. 

We can now deal with the general case,  where $f$ is reducible and not a power of a single irreducible polynomial. This allows us to factorise $f$ in ${\mathbb Z}[t]$ as $f=gh$ where $g$ and $h$ have different irreducible factors. If $\deg g\ne\deg h$ we can replace $f$ with
\[f^*=g^*h^*\quad\hbox{where}\quad g^*=g^{\deg h}\quad\hbox{and}\quad h^*=h^{\deg g},\]
so that $f^*$ takes prime power values at the same integers $t$ as $f$ does. Since $g^*$ and $h^*$ are distinct but have the same degree, we can apply the preceding argument to show that $f^*$ takes prime power values at only finitely many $t\in{\mathbb N}$, and hence the same applies to $f$. Finally, we can extend this result to all $t\in{\mathbb Z}$ either directly as above, using the fact that $g(t)/h(t)$ has similar limiting behaviour when $t\to-\infty$, or by applying the above argument for $t>0$ to $f(-t)$, which factorises in the same way as $f(t)$.
\end{proof}

In particular, let $f=f_{n,r}$ in a case where this polynomial is reducible, or equivalently $n$ is a triangular number $a(a+1)/2$ and $\Delta$ is a non-zero square $4n^2(8n+1)=4n^2(2a+1)^2$. Then $f$ factorises in ${\mathbb Z}[t]$ as 
\[f(t)=g(t)h(t)=\frac{1}{2}(4nt+r+a)(4nt+r-a-1),\]
where the first or second displayed linear polynomial has both of its coefficients even as $r\equiv a$ mod~$(2)$ or not, so that it absorbs the factor $\frac{1}{2}$. In either case, the resulting linear factors $g$ and $h$ of $f$ in ${\mathbb Z}[t]$ are distinct and irreducible, so Theorem~\ref{th:powers} implies that $f(t)$ is a prime power for only finitely many $t\in{\mathbb Z}$.

\begin{prop}\label{prop:red}
If $f_{n,r}$ is reducible and $n>1$ then $f_{n,r}(t)$ is not a prime power for any 
integer $t\ge 1$.
\end{prop}

\begin{proof} Suppose that $f:=f_{n,r}$ is reducible and $n>1$, so $n=a(a+1)/2$ for some integer $a\ge 2$ by Lemma~\ref{le:abc}, and that $f(t)=p^e$ for some prime $p$ and integers $e, t\ge 1$.

\medskip

\noindent{\bf Case 1} If $r\equiv a$ mod~$(2)$ then $f=gh$ where
\[g(t)=2nt+\frac{r+a}{2}\quad\hbox{and}\quad h(t)=4nt+r-a-1.\]
Since $t\ge 1$ we have $g(t), h(t)>1$ so $g(t)=p^i$ and $h(t)=p^j$ for integers $i,j\ge 1$ with $i+j=e$. 
If $i<j$ then
\[\frac{h(t)}{g(t)}=p^{j-i}\ge p\ge2,\]
giving $-a-1\ge a$, which is impossible since $a\ge 1$. Thus $i\ge j$, so $g(t)\ge h(t)$, leading to
\[t\le\frac{3a-r+2}{4n}\le\frac{3a+1}{2a(a+1)}<1\]
(since $r\ge 1$ and $a\ge 2$), against our hypothesis.

\medskip

\noindent{\bf Case 2} If $r\not\equiv a$ mod~$(2)$ then $f=gh$ where
\[g(t)=4nt+r+a\quad\hbox{and}\quad h(t)=2nt+\frac{r-a-1}{2}.\]
As before we have $g(t)=p^i$ and $h(t)=p^j$ for integers $i,j\ge 1$. If $i\le j$ then $g(t)\le h(t)$, leading to\
\[2nt\le \frac{r-a-1}{2}-(r+a)<0,\]
which is impossible. Thus $i>j$, so
\[\frac{g(t)}{h(t)}=p^{i-j}\ge p.\]
If $p^{i-j}=2$ then $g(t)=2h(t)$, giving $a=-a-1$, which is impossible. Hence $p^{i-j}\ge 3$, so $g(t)\ge 3h(t)$, leading to
\[t\le\frac{a-r+3}{4n}\le\frac{a}{4n}=\frac{1}{2(a+1)}<1,\]
again contradicting our hypothesis.
\end{proof}

\begin{rema}\label{re:negative}
Although Theorem~\ref{th:powers} applies to all $t\in{\mathbb Z}$, Proposition~\ref{prop:red} applies only to integers $t\ge 0$ and cannot be extended to the case $t<0$. For example, the polynomial
\[f(t)=f_{3,5}(t)=72t^2+54t+7=(12t+7)(6t+1),\]
satisfies $f(-1)=5^2$, with $g(-1)=h(-1)=-5$. Of course, negative values of $t$ are not relevant to the 2-designs considered in this paper.

The condition $n>1$ is required in Proposition~\ref{prop:red}, since the polynomial
\[f(t)=f_{1,1}(t)=8t^2+2t-1=(2t+1)(4t-1)\]
satisfies $f(1)=3^2$. The block designs $\mathcal D$ considered here all satisfy this condition.
\end{rema}


\section{Values at $t=0$}\label{t=0}

Proposition~\ref{prop:red} leaves open the possibility, which is relevant to 2-designs, that
\[f(0)=f_{n,r}(0)=\frac{r(r-1)}{2}-n\]
could be a prime power. Prime values $f_{n,r}(0)$ seem to arise quite frequently when $f_{n,r}$ is irreducible:
for example, of the twenty polynomials in Table~\ref{tab:up-to-n=9}, sixteen have prime values at $t=0$,
three have the value $1$, and $f_{8,18}$ has the value $145$.
However, the situation is rather different for reducible polynomials $f_{n,r}$,
those for which $n$ is a triangular number $a(a+1)/2$.

\begin{prop}\label{prop:realize}
Let $f_{n,r}$ be reducible, and satisfy (\ref{eq:f}) and (\ref{eq:conditions}). Then $f_{n,r}(0)$ is a prime power 
$p^e$, $e\ge 1$, if and only if $p$ is odd and one of the following occurs:
\begin{itemize}
\item[a)] $e=2i$ is even, with $n=(p^e-1)/8>1$, $a=(p^i-1)/2$ and $r=(3p^i+1)/2$, or
\item[b)] $p^e=7$, with $n=3$, $a=2$ and $r=5$ $($as in\/ {\rm Remark~\ref{re:negative})}.
\end{itemize}
\end{prop}

Note that by (a) every even power $p^e>9$ of an odd prime $p$
can be realised as a value $f_{n,r}(0)$ of a reducible polynomial $f_{n,r}$.

\begin{exam}
One can realise $5^2$ as a value by taking $n=3$, $a=2$ and $r=8$. This gives
\[f(t)=f_{3,8}(t)=72t^2+90t+25=(6t+5)(12t+5)\]
with $f(0)=5^2$. Similarly, one can realise $7^2$ by taking $n=6$, $a=3$ and $r=11$, 
so that
\[f(t)=f_{6,11}(t)=288t^2+252t+49=(12t+7)(24t+7)\]
with $f(0)=7^2$. Taking $n=(13^4-1)/8=3\,570$ and $r=(3\cdot 13^2+1)/2=254$ we get
\[f(t)=f_{n,r}(t)=101\,959\,200\,t^2+3\,619\,980\,t+28\,561 = (7\,140t+169)(14\,280t+169)\]
with $f(0)=28\,561=13^4$.
\end{exam}

\noindent
{\em Proof}\/ of Proposition \ref{prop:realize}.
If we put $t=0$ in Case~(1) of the proof of Proposition~\ref{prop:red}, where $r\equiv a$ mod~$(2)$, we have
\[\frac{r+a}{2}=p^i\quad\hbox{and}\quad r-a-1=p^j\]
for integers $i, j\ge 0$ with $i+j=e\ge 1$ and $i\ge j$. Solving these simultaneous equations gives
\[r=p^i+\frac{p^j+1}{2}\quad\hbox{and}\quad a=p^i-\frac{p^j+1}{2},\]
so that
\[n=\frac{a(a+1)}{2}=\frac{1}{8}\left((2p^i-p^j)^2-1\right).\]
(Recall that $f_{n,r}$ is reducible if and only if $8n+1$ is a perfect square.)
Here we require $p^j$ to be odd, so that $r\equiv a$ mod~$(2)$; however, we reject solutions with $p=2$ and $j=0$ since they give $c=2^i$
and $r(r-1)/2\not\equiv n+1$ mod~$(2n)$, contradicting condition (\ref{eq:conditions}), so $p$ must be an odd prime.

The condition that $r(r-1)/2\equiv n+1$ mod~$(2n)$ also excludes many solutions when $p$ is odd. We have
\[\frac{r(r-1)}{2}-n-1=p^{i+j}-1\quad\hbox{and}\quad 2n=\frac{((2p^i-p^j)^2-1}{4},\]
so if $i>j$ then
\[0<\frac{r(r-1)}{2}-n-1<2n\]
and hence $r(r-1)/2\not\equiv n+1$ mod~$(2n)$. However, if we take $i=j$ then
\[r=\frac{3p^i+1}{2},\quad a=\frac{p^i-1}{2}\quad\hbox{and}\quad n=\frac{p^{2i}-1}{8},\]
so that $r<4n$ provided $p^i>3$, and
\[\frac{r(r-1)}{2}-n-1=p^{2i}-1=8n\equiv 0\; {\rm mod}\, (2n)\]
as required. Thus every even power $p^e=p^{2i}>9$ of an odd prime $p$ is the value
of some reducible polynomial $f_{n,r}$ at $t=0$, giving conclusion (a).

A similar argument applies in Case~(2) of the proof of Proposition~\ref{prop:red}, where $r\not\equiv a$ mod~$(2)$. We now have
\[r+a=p^i\quad\hbox{and}\quad \frac{r-a-1}{2}=p^j,\]
with $i>j$, so that
\[r=\frac{p^i+1}{2}+p^j\quad\hbox{and}\quad a=\frac{p^i-1}{2}-p^j,\]
giving
\[n=\frac{a(a+1)}{2}=\frac{1}{8}\left((p^i-2p^j)^2-1\right).\]
In this case
\[\frac{r(r-1)}{2}-n-1=p^{i+j}-1\quad\hbox{and}\quad 2n=\frac{((p^i-2p^j)^2-1}{4}.\]
We need
\[p^{i+j}-1=\frac{r(r-1)}{2}-n-1\ge 2n=\frac{p^{2i}-1}{4}-p^{i+j}+p^{2j}\]
so that
\[2p^{i+j}\ge\frac{p^{2i}+3}{4}+p^{2j}>\frac{p^{2i}}{4}\]
and hence $p^{i-j}\le 8$. Since $i>j$ and $p\ge 3$ this implies that $i-j=1$ and $p=3, 5$ or $7$. Thus only odd powers $p^e$ of these three primes can arise in Case~2.

Putting $i=j+1$ gives
\[\frac{r(r-1)}{2}-n-1=p^{2j+1}-1
\quad\hbox{and}\quad
2n=\frac{(p^{j+1}-2p^j)^2-1}{4}=\frac{(p-2)^2p^{2j}-1}{4}.\]
Now $2n$ divides $\frac{r(r-1)}{2}-n-1$, so multiplying by $4$ shows that
\[8n=(p-2)^2p^{2j}-1\quad\hbox{divides}\quad 4\left(\frac{r(r-1)}{2}-n-1\right)=4(p^{2j+1}-1)\]
Defining $q:=p^{2j}$, we see that
$(p-2)^2q-1$ divides $4(pq-1)$. We now apply this with $p=3$, $5$ and $7$ in turn.

If $p=3$ then $q-1$ divides $12q-4=12(q-1)+8$, so $q-1$ divides $8$, giving $q=1$ or $9$.
If $q=1$ then $j=0$, giving $r=3$, $a=0$ and $n=0$, whereas we need $n>1$.
If $q=9$ then $j=1$, giving $r=8$, $a=1$ and $n=1$, again too small.
Thus $p\ne 3$.

If $p=5$ then $9q-1$ divides $20q-4=2(9q-1)+2(q-1)$ and hence $9q-1$ divides $2(q-1)$ giving $q=1$.
Then $j=0$, so $r=4$, $a=1$ and $n=1$, whereas we need $n>1$. Thus $p\ne 5$.

If $p=7$ then $25q-1$ divides $28q-4=25q-1+3(q-1)$ and hence $25q-1$ divides $3(q-1)$ giving $q=1$. Then $j=0$, so $r=5$, $a=2$ and $n=3$, with $r<4n$ and $r(r-1)/2-n-1=6\equiv 0$ mod~$(2n)$; this gives the polynomial
\[f(t)=f_{3,5}(t)=72t^2+54t+7=(12t+7)(6t+1)\]
in conclusion (b), with $f(0)=7$. 
\hfill$\Box$


\section{Conclusions}

We have found large numbers of prime values of those polynomials $f_{n,r}$ appearing in~\cite{ADP} for which $n$ is not a triangular number. The numbers found agree closely with the estimates for them provided by Li's recent version of the Bateman--Horn Conjecture. While this does not prove the conjecture in~\cite{ADP} that these polynomials take infinitely many prime values, and thus give infinite families of block designs, it provides strong evidence for this, and it also adds extra support for the validity of the Bunyakovsky and Bateman--Horn Conjectures.


\section{Acknowledgements}

We are grateful to Yuri Bilu, to Cheryl Praeger and to Weixiong Li for many useful comments.
Alexander Zvonkin was  partially supported by the ANR project {\sc Combin\'e} (ANR-19-CE48-0011).



\bigskip

\begin{thebibliography}{99}

\bibitem{AFG} S.~L.~Aletheia-Zomlefer, L.~Fukshansky and S.~R.~Garcia, The Bateman--Horn 
	conjecture: heuristics, history, and applications, {\sl Expo. Math.} 38 
	(2020), 430--479. Also available at \url{arXiv-math[NT]:1807.08899v4}. %
	
\bibitem{ADP} C.~Amarra, A.~Devillers and C.~E.~Praeger, Delandsheer--Doyen parameters 
	for block-transitive point-imprimitive block designs, \url{arXiv-math[CO]:2009.00282}. %
	
\bibitem{BPS} W.~D.~Banks, F.~Pappalardi and I.~E.~Shparlinski
 On group structures realized by elliptic curves over arbitrary finite fields,
{\sl Exp.~Math.} 21 (2012), 11--25.

\bibitem{BH} P.~T.~Bateman and R.~A.~Horn, A heuristic asymptotic formula concerning the 
	distribution of prime numbers, {\sl Math.~Comp.} 16 (1962), 220--228. %
	
\bibitem{BS} Z.~I.~Borevich and I.~R.~Shafarevich, {\sl Number Theory}, Academic Press, 
	New York and London, 1966. %
	
\bibitem{Bun-1857} V. Bouniakowsky, Sur les diviseurs num\'eriques invariables des 
	fonctions rationnelles enti\`eres, {\em M\'em. Acad. Sci. St. P\'eteresbourg}, 
	$6^{\rm e}$ s\'erie, vol. VI (1857), 305--329.\footnote{Numerous publications
	give the following wrong title for Bunyakovsky's paper: ``Nouveaux th\'eor\`emes 
	relatifs \`a la distinction des nombres premiers et \`a la d\'ecomposition des 
	entiers en facteurs''. According to the French Wikipedia (see~\cite{Bun-wiki}), 
	an article with this title does indeed exist, but it was published in 1840 and 
	not in 1857, and it does not discuss the conjecture in question. The reader 
	may also consult the original paper reproduced in the Google archive.} Available at: 
	\url{https://books.google.fr/books?hl=fr&id=wXIhAQAAMAAJ&pg=PA305#v=onepage&q&f=false}. %
	
		
\bibitem{BG} J.~Boxall and D.~Gruenewald, Heuristics on pairing-friendly abelian varieties,
{\sl LMS J.~Comput.~Math.} 18 (2015), 419--443.
	
\bibitem{Bun-wiki} Bunyakovsky conjecture, Wikipedia, 
	\url{https://en.wikipedia.org/wiki/Bunyakovsky_conjecture}.
	Conjecture de Bouniakovsky, Wikip\'edia, 
	\url{https://fr.wikipedia.org/wiki/Conjecture_de_Bouniakovski}. %
	
\bibitem{Coh} H.~Cohen, High-precision computation of Hardy--Littlewood constants,
	 Available at \url{https://oeis.org/A221712/a221712.pdf}. %

\bibitem{Cook} J.~D.~Cook, Distribution of prime powers, 
	John D. Cook's blog,
	\url{https://www.johndcook.com/blog/2018/09/03/counting-prime-powers}. %

\bibitem{CT} S.~Covanov and E.~Thom\'e, Fast integer multiplication using generalized Fermat primes,
	{\sl Math.~Comp.} 8 (2019), 1449--1477.
	
\bibitem{DS} C.~David and E.~Smith, A Cohen--Lenstra phenomenon for elliptic curves,
{\sl J.~London Math.~Soc.~(2)} 89 (2014), 24--44.

\bibitem{DD} A.~Delandtsheer and J.~Doyen, Most block-transitive $t$-designs are 
	point-primitive, {\sl Geom.~Dedicata} 29 (1989), 307--310. %
	
\bibitem{EKN} D.~Ellis, G.~Kalai and B.~Narayanan, On symmetric intersecting families, {\sl European J.~Combin.} 86 (2020), 103094.

	\bibitem{Eul} L.~Euler, letter to Goldbach, 28th October 1752 (letter CXLIX), available at
	\url{http://eulerarchive.maa.org/correspondence/letters/OO0877.pdf}. 
	See also De numeris primis valde magnis, 
	{\sl Novi Commentarii academiae scientiarum Petropolitanae\/} 9, 99--153 (1760);
	reprinted in {\sl Commentat.~arith.\/}~1, 356--378, 1849, and in 
	{\sl Opera Omnia: Ser.~1, vol.~3}, 1--45.

\bibitem{HL} G.~H.~Hardy and J.~E.~Littlewood,
	Some problems of `Partitio numerorum'; III: On the expression of a number as 
	a sum of primes, {\sl Acta Math.} 114 (1923), 215--273. %
	
\bibitem{JW} M. J. Jacobson, Jr., and H. G. Williams, New quadratic polynomials with high
	densities of prime values, {\sl Math.~Comp.}\/ 72 (2002), no.~241,
	499--519. 

\bibitem{JJ} G.~A.~Jones and J.~M.~Jones, {\sl Elementary Number Theory}, Springer, 1998. %

\bibitem{JZ21} G.~A.~Jones and A.~K.~Zvonkin, Klein's ten planar dessins of degree 11, 
	and beyond, \url{https://arxiv.org/pdf/2104.12015.pdf}. 
	To appear. %

\bibitem{JZ21a} G.~A.~Jones and A.~K.~Zvonkin, Projective primes and the Bateman--Horn Conjecture,
\url{https://arxiv.org/pdf/2106.00346.pdf}.
	

\bibitem{Kim} D.~Kim, Nonexistence of perfect 2-error-correcting Lee codes in certain dimensions,
{\sl European J.~Combin.} 63 (2017), 1--5.

\bibitem{W.Li} W.~Li, A note on the Bateman--Horn conjecture,
	{\sl J.~Number Theory} 208 (2020), 390--399.
	Also available at \url{https://arxiv.org/pdf/1906.03370.pdf}. %

\bibitem{NZM} I.~Niven, H.~S.~Zuckerman and H.~L.~Montgomery, {\sl An Introduction to 
	the Theory of Numbers} (5th ed.), Wiley, New York, 1991. %

\bibitem{OEIS-A006880} The Online Encyclopedia of Integer Sequences,
	Entry A006880, \url{https://oeis.org/A006880}.

\bibitem{Rivin} I.~Rivin, Some experiments on Bateman--Horn, 2015,
	\url{https://arxiv.org/pdf/1508.07821}.
	
\bibitem{Scholl17} T.~Scholl, Isolated elliptic curves and the MOV attack, {\sl J.~Math.~Cryptol.} 11 (2017), 131--146.

\bibitem{Scholl19} T.~Scholl, Super-isolated elliptic curves and abelian surfaces in cryptography, {\sl Exp.~Math.} 28 (2019), 385--397.

\bibitem{Sha} M.~Sha, Heuristics of the Cocks--Pinch method, {\sl Adv.~Math.~Commun.} 8 (2014), 103--118.
	
\bibitem{ShL} D. Shanks and M. Lal, Bateman's constant reconsidered and the distribution
	of cubic residues, {\sl Math.~Comp.}\/ 26 (1972), no. 117, 265--285.
 
\end{thebibliography}
\end{document}